\documentclass[12pt,reqno,oneside]{amsart}
      
\usepackage{amsmath,amsthm,amsfonts,amssymb}
\usepackage{indentfirst}
\usepackage{url}
\usepackage{graphicx}
\usepackage[usenames]{color}

\theoremstyle{plain}
\newtheorem{theorem}{Theorem}[section]

\newtheorem{prop}[theorem]{Proposition}

\theoremstyle{definition}

\newtheorem{obs}[theorem]{Remark}

\numberwithin{equation}{section}

\begin{document}

\baselineskip=18pt

\title[Growth Models Under Uniform Catastrophes]{Growth Models Under Uniform Catastrophes}
\author[Joan  Amaya]{Joan Amaya}
\address[Joan J. Amaya]{Statistics Department, Institute of Mathematics and Statistics, University of S\~ao Paulo, CEP 05508-090, S\~ao Paulo, SP, Brazil.}
\email{jjamayat@ime.usp.br}
\author[Valdivino V. Junior]{Valdivino V. Junior}
\address[Valdivino V. Junior]{Institute of Mathematics and Statistics, Federal University of Goias, Campus Samambaia, 
CEP 74001-970, Goi\^ania, GO, Brazil}
\email{vvjunior@ufg.br}
\author[F\'abio P. Machado]{F\'abio P. Machado}
\address[F\'abio P. Machado]{Statistics Department, Institute of Mathematics and Statistics, University of S\~ao Paulo, CEP 05508-090, S\~ao Paulo, SP, Brazil.}
\email{fmachado@ime.usp.br}

\author[Alejandro Rold\'an-Correa]{Alejandro~Rold\'an-Correa}
\address[Alejandro Rold\'an]{Instituto de Matem\'aticas, Universidad de Antioquia, Calle 67, no 53-108, Medellin, Colombia}
\email{alejandro.roldan@udea.edu.co}

\thanks{F\'abio Machado and Valdivino Junior were partially supported by Fapesp (2023/13073-8) and CNPq (408590/2024-6). Alejandro Roldan by Universidad de Antioquia (Project No. 2025-80410).}
%corrigido em 27/nov - Fabio

%%%%

\date{\today}

\begin{abstract}
We consider stochastic growth models for populations organized in colonies and subject to uniform catastrophes. To assess population viability, we analyze scenarios in which individuals adopt dispersion strategies after catastrophic events. For these models, we derive explicit expressions for the survival probability and the mean time to extinction, both with and without spatial constraints.
In addition, we complement this analysis by comparing uniform catastrophes with binomial and geometric catastrophes in models with dispersion and no spatial restrictions. Here, the terms uniform, binomial and geometric refer to the probability distributions governing the number of individuals that survive immediately after a catastrophe. This comparison allows us to quantify the impact of different types of catastrophic events on population persistence.
\end{abstract}

\maketitle

\section{Introduction}
\label{S: Introduction}

Stochastic population models are essential tools in mathematical biology and ecology for understanding how populations evolve over time under the influence of random events. Among the most significant biological and environmental factors affecting population dynamics are catastrophes and spatial constraints, which can drastically reduce population size or even drive populations to extinction. When catastrophes occur, survivors may adopt adaptive strategies to enhance their survival chances. For example, some colonial species employ dispersal strategies, spreading individuals out to establish new colonies. The primary motivation for studying random catastrophes in these models is the need to better understand and predict the effects of catastrophic events on population dynamics.

A catastrophe can either decimate an entire population or affect only part of it. To model such events, it is generally assumed that when a catastrophe strikes, population size is reduced according to some probabilistic law. Models like those in \cite{AEL2007, BGR1982, B1986, CairnsPollett, KPR2016} extend traditional birth-death processes by incorporating catastrophic events, with surviving individuals remaining together as a single colony. In particular, Brockwell et al. \cite{BGR1982} introduce birth and death processes involving various types of catastrophes, such as binomial, geometric and uniform catastrophes. Here, the terms uniform, binomial and geometric refer to the probability distributions governing the number of individuals that survive immediately after a catastrophe. These studies primarily explore the probability of extinction and the expected time to extinction.

Several models have been proposed to analyze various types of catastrophes and dispersal mechanisms (see \cite{DJMR2023, JMR2016, JMR2020, JMR2023, MRV2018, MRS2015, S2014}), including binomial and geometric catastrophes. In this work, we study structured populations residing in colonies that experience uniform catastrophes. We analyze two strategies: in the first, after such catastrophic events, the surviving individuals remain together in the same colony. In the second scenario, the survivors disperse in an effort to establish new colonies.

This paper is organized into four sections. In Section~2, we introduce three models describing population growth under uniform catastrophes, considering scenarios both with and without dispersion. We also present results on the survival probability and the mean time to extinction. In Section~3, we discuss these results in detail and compare uniform catastrophes with binomial and geometric catastrophes in models with dispersion. Finally, in Section~4, we provide the proofs of the results presented in Section~2.

\section{Models and Results}
\subsection{Uniform Catastrophes}
Populations are often exposed to catastrophic events that result in the massive elimination of individuals. Uniform catastrophes are sudden, large-scale occurrences that instantaneously and uniformly impact a significant fraction of the population. Typical examples include natural disasters such as wildfires, hurricanes, epidemics, and oil spills. These events can drastically reduce the population size, affecting its recovery capacity and future growth. They often have devastating consequences and are difficult to predict and manage. To model such events, it is assumed that when a population experiences a catastrophe, its size is reduced according to a discrete uniform probability distribution. Specifically, if the population size at the time of the catastrophe is \(i\), it is reduced to \(j\) with probability 
\[ \mu_{ij} = \frac{1}{i}, \hspace{0.5cm} j\in\{0,1,\ldots,i-1\}.\]
The form of \(\mu_{ij}\) represents what is called a \textit{uniform catastrophe}, see Brockwell~\textit{et al}~\cite{BGR1982}.

In the following subsections, we introduce three models in which the population of individuals living in colonies grows according to a Poisson process with rate \( \lambda > 0 \), while uniform catastrophes occur according to a Poisson process with rate 1. Additionally, we assume that each colony starts with a single individual. In the first model, after each catastrophic event, the surviving individuals remain together in the same colony (no dispersion). In the second and third models, the survivors disperse and attempt to establish new colonies. The key difference between these two models lies in spatial constraints: in the second model, certain locations are unavailable for new colonies, whereas in the third model, no such restrictions exist.  The following result characterizes the distribution of the number of individuals that survive a uniform catastrophe.

\begin{prop}\label{distribuiçao_individuos}
Let $N$ be the number of individuals that survive an uniform catastrophe.   
\begin{eqnarray*}
  \mathbb{P}(N=n)&=&\dfrac{1}{\lambda+1}\left(\dfrac{\lambda}{\lambda+1}\right)^n \Phi\left(\dfrac{\lambda}{\lambda+1}, 1, n+1\right),\\
\mathbb{E}(s^N)&=&\dfrac{\ln(1+\lambda(1-s))}{\lambda(1-s)}\quad \text{ and }\\
\mathbb{E}(N)&=&\dfrac{\lambda}{2},  
\end{eqnarray*}

where $\Phi$ is the Lerch Transcendent function, defined by 
\begin{equation}\label{DefLerchFunction}
   \Phi(z, s, a) = \sum_{j=0}^{\infty} \frac{z^j}{(j + a)^s}, \quad |z|<1, \quad a\neq 0,-1,\ldots 
\end{equation}
\end{prop}

See Gradshteyn and Ryzhik \cite[Eq. 9.550]{GR2007} for the definition of the Lerch transcendent function. Sometimes, we use its integral representation \cite[Eq. 9.556]{GR2007}, given by
\begin{equation}\label{IntLerchFunction}
\Phi(z, s, a) = \frac{1}{\Gamma(s)}\int_{0}^{\infty} \frac{t^{s-1} e^{-(a-1)t}}{e^t - z} dt. 
\end{equation}

\subsection{Growth model without dispersion}

Consider a population of individuals residing in a colony that experiences uniform catastrophes. After each catastrophic event, the surviving individuals remain together within the same colony. We assume that the colony starts with a single individual and that it generates new individuals at a rate \(\lambda > 0\), while uniform catastrophes occur at rate 1.  Specifically, we consider the population size (i.e., the number of individuals in the colony) at time \( t \) as a continuous-time Markov process \( \{X(t) : t \geq 0\} \), with $X(0)=1$, which we denote by \( C(\lambda) \). The infinitesimal generator of $X(t)$ is given by  
\[
q_{ij} =
\begin{cases}
\lambda, & j = i + 1, \\
\frac{1}{i}, & j = 0, \dots, i - 1, \\
-(\lambda + 1), & i = j, \\
0, & \text{otherwise}.
\end{cases}
\]

 We say that the process $C(\lambda)$ \textit{survives} if, with positive probability, individuals are present at all times; that is, $X(t) \neq 0$ for all $t > 0$. Otherwise, we say that $C(\lambda)$ \textit{dies out}.
In the following theorem, we prove that the process $C(\lambda)$ becomes extinct almost surely.

\begin{theorem}\label{Th:non-disp}  
The process \( C(\lambda) \) dies out  for all \( \lambda > 0 \).
\end{theorem}

\subsection{Growth models with dispersion and spatial restriction}
Let $\mathbb{T}_d^+$ be an infinite rooted tree whose vertices
have degree $d+1$, except the root that has degree $d$. Let us define a process with 
dispersion on $\mathbb{T}_d^+$, starting from a single colony placed at the root of $\mathbb{T}_d^+$, with just one individual. The number of individuals in a colony grows following a Poisson process of rate $\lambda>0$. To each colony 
we associate an exponential time of mean 1 that indicates when the uniform catastrophe strikes a colony.
Each individual that survives the catastrophe randomly picks a neighbor vertex between the $d$ neighboring vertices furthest from the root to create new colonies. Among the survivors that go to the same vertex to create a new colony at it, only one succeeds, the others die. So in this case  when a catastrophe occurs in a colony, that colony is replaced by 0,1, ... or $ d $ colonies. Let us denote this process with by $C_d(\lambda)$.

The process \( C_d(\lambda) \) is a continuous-time Markov process with state space \( \{0,1,2,\ldots\}^{\mathbb{T}_d^+} \). We say that the process \emph{survives} if, with positive probability, there are colonies at all times. Otherwise, we say that the process \emph{dies out}. We denote by \( \psi_d \) the extinction probability of the processes  \( C_d(\lambda) \). Using coupling arguments, we can show that  \( \psi_d \) is a non-increasing functions of \( d \) and \( \lambda \). The following results establish conditions on the parameters \(\lambda\) and \(d\) for the survival-extinction transition phase, as well as exact calculations for the extinction probability and the mean time to extinction when \(d=2\) and \(d=3\).

\begin{theorem}\label{condicao_extincao_uni_indepen}
	The process \( C_d(\lambda) \) survives (\( \psi_d < 1 \)) if and only if
	\begin{equation*}
\dfrac{d^2}{d-1}\ln\left(\dfrac{\lambda+d}{d}\right)<\lambda.		
	\end{equation*}
 
\end{theorem}

\begin{theorem}\label{uniforme_independente2}
	The process \( C_2(\lambda) \) survives (\( \psi_2 < 1 \)) if and only if
	\[
	4\ln\left(1+\frac{\lambda}{2}\right)<\lambda.
  %\lambda > -4W\left(-\dfrac{1}{2e^{1/2}}\right)-2. 
	\]
 	Additionally,
	\begin{equation*}
		\psi_2=\min\left\{ 1, \dfrac{\ln(\lambda+1)}{\lambda+\ln(\lambda+1)-4\ln\left(1+\lambda/2\right)}\right\}
	\end{equation*}
\end{theorem}

\begin{theorem}\label{uniforme_independente3}
The process \( C_3(\lambda) \) survives (\( \psi_3 < 1 \)) if and only if
\[
\dfrac{9}{2} \ln\left(1+\dfrac{\lambda}{3}\right)<\lambda.
\]
%\[\lambda > -\dfrac{9}{2} W\left(-\dfrac{1}{3e^{2/3}}\right) - 3.\]
Additionally,
\[
\psi_3 = \min\left\{1,  \frac{1 + 2a - 9b + \sqrt{(1 - 9b)^2 + 4a(2 - 9c)}}{2(a - 9b + 9c - 1)} \right\},
\]
where
\[
a = \frac{\ln(1+\lambda)}{\lambda}, \quad
b = \frac{1}{2\lambda}\ln\!\left(1+\frac{2\lambda}{3}\right), \quad
c = \frac{1}{\lambda}\ln\!\left(1+\frac{\lambda}{3}\right).
\]
 
\end{theorem}

In the following result, we obtain the mean time to extinction when the processes \( C^i_d(\lambda) \) almost surely die out, for \( d=2 \) and \( d=3 \). 

\begin{theorem}\label{esperanca_indepen}
Let \( \tau_d \) be the extinction time of the process \( C_d(\lambda) \).

\begin{enumerate}
\item If \( 4\ln\left(1+\frac{\lambda}{2}\right)>\lambda \), then 
\begin{align*}
			\mathbb{E}[\tau_2]&=\dfrac{1}{\beta}\ln\left(\dfrac{\alpha}{\alpha-\beta}\right).
		\end{align*}
where 
\[
\alpha = \frac{\ln(\lambda + 1)}{\lambda} \quad \text{and} \quad
\beta = 1 - \frac{1}{\lambda}\,\ln\!\left(\frac{(\lambda + 2)^4}{16(\lambda + 1)}\right).
\] 
If \( 4\ln\left(1+\frac{\lambda}{2}\right)=\lambda \), then \( \mathbb{E}[\tau_2] = \infty \).\\
    
    \item If \( \frac{9}{2} \ln\left(1+\frac{\lambda}{3}\right)>\lambda \), then 
    \[\mathbb{E}[\tau_3]=\dfrac{1}{\sqrt{4\alpha \gamma+\theta^2}}\ln\left[ \dfrac{2\alpha-\theta+\sqrt{4\alpha \gamma+\theta^2}}{2\alpha-\theta-\sqrt{4\alpha \gamma+\theta^2}} \right],\]
  where \[
\theta =1 - \frac{1}{2\lambda}\,\ln\!\left(\frac{(2\lambda + 3)^9}{27^3(\lambda + 1)^4 }\right)  \quad \text{and} \quad
\gamma = 1 - \frac{1}{2\lambda}\,\ln\!\left(\frac{(\lambda + 3)^{18}(\lambda + 1)^2}{27^3(2\lambda + 3)^9}\right).
\] 
   If \( \dfrac{9}{2} \ln\left(1+\dfrac{\lambda}{3}\right)=\lambda\), then \( \mathbb{E}[\tau_3] = \infty \).
\end{enumerate}
 
\end{theorem}

\begin{obs}
We define the critical parameter for the survival–extinction phase transition by
\[
    \lambda_d \;=\; \inf\{\lambda>0 : C_d(\lambda) \ \text{survives}\}.
\]
Numerical approximations derived from Theorem~\ref{condicao_extincao_uni_indepen} yield the values reported in Table~\ref{table:lambda_d}. These values illustrate the monotone decrease of \( \lambda_d \) as a function of the dimension \( d \), as well as the convergence of the critical parameter toward the limiting value \( \lambda_\infty = 2 \) when \( d \to \infty \). This limiting behavior is consistent with the model with dispersion but without spatial constraints, which we present in the following section.

\begin{table}[h]
\centering
\begin{tabular}{|c|ccccccccc|}
\hline
$d$
& 2 & 3 & 5 & 7 & 10 & 20 & 50 & 100 & 200 \\
\hline
$\lambda_d$
& 5.026 & 3.432 & 2.693 & 2.456 & 2.302 & 2.133 & 2.053 & 2.027 & 2.013 \\
\hline
\end{tabular}
\caption{Critical parameters, rounded to three decimal places.}
\label{table:lambda_d}
\end{table}
\end{obs}

\subsection{Growth model with dispersion but no spatial restrictions.} 
Consider a population of individuals divided into separate colonies. Each colony begins with  
an individual. The number of individuals in each colony increases independently according 
to a Poisson process of rate $\lambda > 0 $.
To each colony we associate an exponential time of mean 1 that indicates when the uniform catastrophe strikes a colony. Each individual that survived the catastrophe 
begins a new colony independently of everything else.
We denote this process by $C_*(\lambda)$ and consider it starting from a single colony with just one individual.\\

We say that the process $C_*(\lambda)$ {\it survives}  if there is a positive probability that colonies exist at any time. Otherwise, we say that the process {\it dies out }.

\begin{theorem}\label{th:comdisptime}
Let \( \psi_* \) and \( \tau_* \) denote the extinction probability and the extinction time of the process \( C_*(\lambda) \), respectively.

\begin{enumerate}

    \item For \( \lambda > 2 \), the extinction probability \( \psi_* \) is given by the smallest non-negative solution of
    \[
        \ln\!\bigl(1+\lambda(1-s)\bigr) = \lambda s(1-s).
    \]
    For \( \lambda \le 2 \), the process becomes extinct almost surely, and thus \( \psi_* = 1 \).

    \item For \( \lambda \le 2 \), the expected extinction time is
    \[
        \mathbb{E}[\tau_*]
        = \frac{1}{\lambda}\int_{0}^{\lambda}
        \frac{x^{2}}{\lambda \ln(1+x) - x(\lambda - x)}\, dx .
    \]

\end{enumerate}
\end{theorem}

\section{Discussion}

\subsection{Dispersion as a survival strategy.}

The results obtained in this work highlight several qualitative features of growth models subject to catastrophic events and allow us to assess the effectiveness of dispersion as a mechanism that enhances long-term population persistence. Theorems~\ref{Th:non-disp}, \ref{condicao_extincao_uni_indepen}, \ref{esperanca_indepen}, and \ref{th:comdisptime} show that dispersion plays a fundamental role in preventing extinction under uniform catastrophes. In the absence of dispersion (Theorem~\ref{Th:non-disp}), the population becomes extinct almost surely for every value of the growth rate~$\lambda$, and this extinction takes place in finite mean time. When dispersion is allowed, however, the population may survive with positive probability, depending on the interaction between the growth rate~$\lambda$ and the spatial structure of the environment.

For populations living on homogeneous trees (Theorems~\ref{condicao_extincao_uni_indepen}--\ref{esperanca_indepen}), dispersion becomes an effective strategy for achieving a positive probability of survival whenever the parameters \(d\) and \(\lambda\) satisfy the inequality characterized in Theorem~\ref{condicao_extincao_uni_indepen}. The explicit computations for \(d = 2\) and \(d = 3\) (Theorems~\ref{uniforme_independente2}, \ref{uniforme_independente3} and \ref{esperanca_indepen}) further illustrate how dispersion reduces the risk of extinction by increasing both the survival probability and the mean time to extinction.

When spatial constraints are removed (Theorem~\ref{th:comdisptime}), dispersion becomes even more efficient. In this regime, survival is possible precisely when $\lambda > 2$, and the critical value can be computed explicitly. The monotone convergence of the critical parameters $\lambda_d$ toward~2 as $d \to \infty$ (Table~\ref{table:lambda_d}) shows that high-dimensional environments behave similarly to the unrestricted model, further reinforcing the idea that dispersion substantially mitigates the impact of uniform catastrophes. Altogether, these results demonstrate that dispersion is a robust mechanism for extending population persistence and increasing resilience to catastrophic events.

\subsection{A comparative analysis of catastrophe types.}

We complement our analysis by comparing uniform catastrophes with geometric and binomial types of catastrophes in models with dispersion but no spatial restrictions. For each class of catastrophes, the corresponding extinction probabilities can be computed explicitly, which allows us to quantify the impact that each type of catastrophic event has on population persistence.

$\bullet$ \textit{Model with geometric catastrophes.}
Consider the model described in Section~2.4, but instead of uniform catastrophes we now assume geometric catastrophes. That is, if the colony size at the time of the catastrophe is $i$ individuals, then $j$ individuals survive with probability
\[
    \mu_{ij}^G =
    \begin{cases}
        (1-p)^{\,i}, & j = 0,\\[4pt]
        p\, (1-p)^{\,i-j}, & 1 \le j \le i.
    \end{cases}
\]
For this model, the probability of extinction is given by
\[
    \psi_G(\lambda,p)
    = \min\left\{
        \frac{(1-p)(\lambda+1)}{\lambda(1+\lambda p)},\; 1
      \right\},
\]
see \cite[Theorem~2.4 and Remark~2.6]{JMR2016}.\\

In the case of uniform catastrophes, Theorem~\ref{th:comdisptime} shows that the extinction probability is given by
\[
\psi_*(\lambda)=
\begin{cases}
1, & \lambda < 2, \\[1em]
\min\{\,s>0 : \ln[1+\lambda(1-s)]=\lambda s(1-s)\,\}, & \lambda \ge 2.
\end{cases}
\]
In contrast with the geometric case, represented by $\psi_G(\lambda,p)$, Figure~\ref{figura1} illustrates that uniform catastrophes are more severe than geometric catastrophes, leading to a substantially higher extinction probability and making population survival more difficult.

\begin{figure}[ht]
	\includegraphics[trim={0cm 0cm 0cm 0cm}, clip, width=15cm]{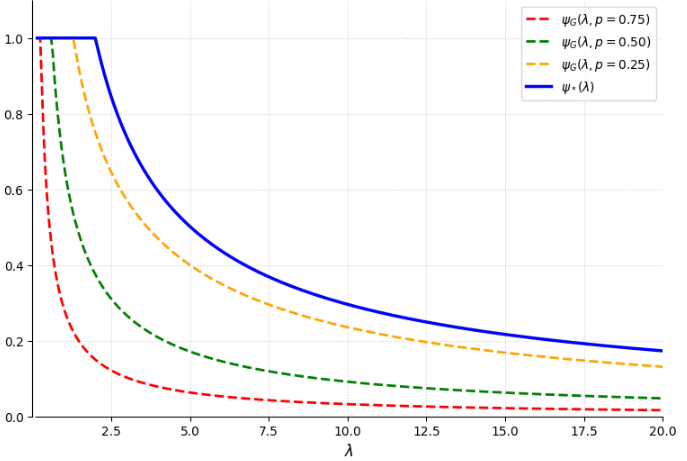}	\caption{Comparison between extinction probabilities in models with uniform catastrophes and geometric catastrophes.} 
	\label{figura1}
\end{figure}

$\bullet$ \textit{Model with binomial catastrophes.}
Consider the model described in Section~2.4, but instead of uniform catastrophes we now assume binomial catastrophes. That is, if the colony size at the time of the catastrophe is $i$ individuals, then $j$ individuals survive with probability
\[ \mu_{ij}^B ={ i \choose j} p^j (1-p)^{i-j}, \ 0\leq j\leq i.\]
For this model, the probability of extinction is given by
\[\psi_B(\lambda,p) = \min\left\{ \frac{1-p}{\lambda p},\, 1 \right\}\]
see \cite[Theorem~2.4 and Remark~2.6]{JMR2016}.

In contrast with the uniform case, represented by $\psi_*(\lambda)$, Figure~\ref{figura2} shows that uniform catastrophes are more severe than binomial catastrophes whenever $p \ge 1/3$. For $p < 1/3$, the curves suggest the existence of a phase transition in the parameter~$\lambda$: if $\lambda \le \lambda_{\bullet}(p)$, uniform catastrophes appear to be less severe than binomial catastrophes, whereas if $\lambda > \lambda_{\bullet}(p)$, uniform catastrophes become more severe than their binomial counterparts. For $p = 0.25$, numerical computations indicate that $\lambda_{\bullet}(0.25) \approx 10.58$.

\begin{figure}[ht]
	\includegraphics[trim={0cm 0cm 0cm 0cm}, clip, width=15cm]{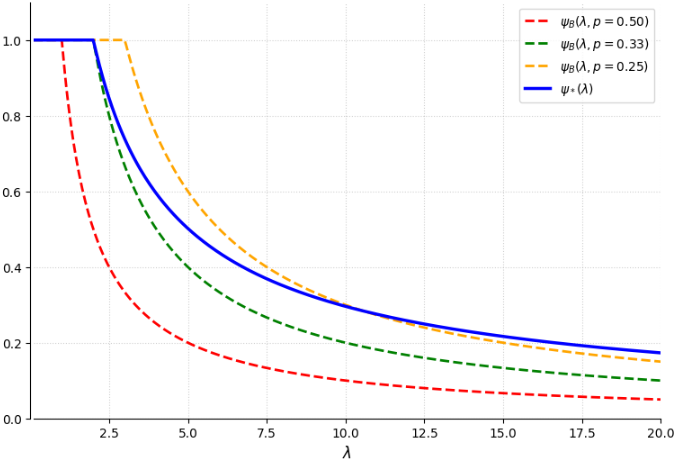}	\caption{ Comparison between extinction probabilities in models with uniform catastrophes and binomial catastrophes.} 
	\label{figura2}
\end{figure}

\section{Proofs}
We begin with the proof of Proposition~\ref{distribuiçao_individuos}. In Section 3.1, we establish the results for the model without dispersion. In Section 3.2, we derive the results for the model with dispersion and spatial constraints. Finally, in Section 3.4, we prove the results for the model with dispersion and no spatial constraints.

\begin{proof}[Proof of Proposition~\ref{distribuiçao_individuos}]
Let $X_t$ be the number of individuals in a colony at time $t$ after its creation.  Since $X_0=1$ and the population grows according to a Poisson process of rate $\lambda$, we have that
\begin{equation*}
\mathbb{P}(X_t=j)=\dfrac{e^{-\lambda t}(\lambda t)^{j-1}}{(j-1)!}, \ \ j\in\{1,2,\ldots\}.
\end{equation*}
When a uniform catastrophe occurs in the colony, it affects the individuals such that, if there are $j$ individuals at the time of the catastrophe, with probability $1/j$, either 0, 1, 2, ..., or $j-1$ individuals survive. Thus, the number of survivors at time $T$ (catastrophe time) in that colony satisfies
	\begin{equation*}
		\mathbb{P}(N=n|X_T=j)=\dfrac{1}{j}, \ \ \ n\in\{0,1,2,\ldots, j-1\}.
	\end{equation*}
Next we consider the distribution of the number of survivals at catastrophe times.	
 
	\begin{eqnarray}\label{DensityN}
		\mathbb{P}(N=n)&=&\int_0^\infty e^{-t}\sum_{j=n+1}^\infty \dfrac{e^{-\lambda t}(\lambda t)^{j-1}}{(j-1)!}\dfrac{1}{j} dt \nonumber\\
		%&=&\int_0^\infty e^{-t}\sum_{j=n}^\infty \dfrac{e^{-\lambda t}(\lambda t)^j}{(j+1)!}dt\nonumber\\
		&=&\sum_{j=n}^\infty \dfrac{\lambda^j}{(j+1)!}\int_0^\infty e^{-t(\lambda+1)}t^j dt \nonumber\\
		%&=&\sum_{j=n}^\infty \dfrac{\lambda^j}{(j+1)!}\dfrac{\Gamma(j+1)}{(\lambda+1)^{j+1}} \nonumber\\
		%&=&\dfrac{1}{\lambda+1}\sum_{j=n}^\infty \dfrac{1}{j+1}\left(\dfrac{\lambda}{\lambda+1}\right)^j \nonumber\\
        &=&\dfrac{1}{\lambda+1}\left(\dfrac{\lambda}{\lambda+1}\right)^n\sum_{j=0}^\infty \dfrac{1}{j+n+1}\left(\dfrac{\lambda}{\lambda+1}\right)^j \nonumber\\
        &=&\dfrac{1}{\lambda+1}\left(\dfrac{\lambda}{\lambda+1}\right)^n \Phi\left(\dfrac{\lambda}{\lambda+1}, 1, n+1\right).
		\end{eqnarray}
When $n=0$, the last line is equals to $\dfrac{\ln(\lambda+1)}{\lambda}$, which is obtained using the identity
$\ln(1-x)=-\sum_{j=1}^\infty \frac{x^{j}}{j},$ for $|x|<1.$ Thus, for the calculation of the moment generating function of \( N \), we have that  

\begin{equation}\label{fgm_aux}
    \mathbb{E}(s^N) =\dfrac{\ln(\lambda+1)}{\lambda}+\dfrac{1}{\lambda+1}\sum_{n=1}^\infty  \left(\dfrac{s \lambda}{\lambda+1}\right)^n \Phi\left(\dfrac{\lambda}{\lambda+1}, 1, n+1\right).
\end{equation}

Letting \( x = \dfrac{\lambda}{\lambda+1} \) and using the integral representation of the Lerch function (\ref{IntLerchFunction}), note that

	\begin{align*}
		\sum_{n=1}^\infty  \left(\dfrac{s \lambda}{\lambda+1}\right)^n \Phi\left(\dfrac{\lambda}{\lambda+1}, 1, n+1\right)
		&=\sum_{n=1}^\infty s^nx^n\Phi(x, 1, n+1)\\
		&=\sum_{n=1}^\infty (sx)^n\int_0^\infty \dfrac{e^{-(n+1)t}}{1-xe^{-t}}dt\\
		&=\int_0^\infty \dfrac{e^{-t}}{1-xe^{-t}}\sum_{n=1}^\infty (xse^{-t})^n dt\\
		&=\dfrac{s\ln(1-x)-\ln(1-sx)}{(s-1)x}.
	\end{align*}
 Thus, substituting this into (\ref{fgm_aux}) and simplifying, we obtain
	\begin{equation*}
		\mathbb{E}(s^N)=\dfrac{\ln(1-\lambda(s-1))}{\lambda(1-s)}.
	\end{equation*}
Therefore,
  \begin{equation*}
		\dfrac{d}{ds} \mathbb{E}(s^N)=\dfrac{[\lambda(s-1)-1]\ln(1+\lambda(1-s))+\lambda(1-s)}{\lambda(s-1)^2[\lambda(s-1)-1]},
	\end{equation*}
and 
	\begin{equation*}
		\mathbb{E}(N)=\lim_{s\rightarrow1^-}\dfrac{d}{ds}\mathbb{E}(s^N)=\dfrac{\lambda}{2}.
	\end{equation*}
\end{proof}

\subsection{Growth model without dispersion}
To establish Theorem~\ref{Th:non-disp}, we apply Foster’s theorem, stated below. A proof of Foster’s theorem can be found in Fayolle \textit{et al.}~\cite[Theorem 2.2.3]{FMM1995}.

\begin{theorem}[Foster's theorem] Let $\{W_n\}_{n\geq 0}$  be an irreducible and aperiodic Markov chain on countable state space $\mathcal{A}=\{\alpha_i,\ i\geq0\}.$ Then, $\{W_n\}_{n\geq 0}$ is ergodic if and only if there exists a positive function $f(\alpha), \ \alpha\in\mathcal{A},$ a number $\epsilon>0$ and a finite set $A\subset\mathcal{A}$ such that 
$$\mathbb{E}[f(W_{n+1})-f(W_{n}) \ | \ W_n=\alpha_j]\leq -\epsilon, \quad \alpha_j\notin A,$$ 
$$\mathbb{E}[f(W_{n+1}) \ | \ W_n=\alpha_i] < \infty, \quad \alpha_i\in A.$$
\end{theorem}

\begin{proof}[Proof of Theorem~\ref{Th:non-disp}]
Let $\{Y_n\}_{n\geq 0}$ be a discrete-time Markov chain embedded on $C(\lambda)$ with transition 
probabilities given by
$$\begin{array}{ll}
P_{i,i+1}=\displaystyle\frac{\lambda}{\lambda +1},&  \ i\geq 0, \\ \\
P_{i,j}=\displaystyle\frac{1}{i(\lambda +1)}, &  \  0\leq j\leq i-1.\\
\end{array}$$

Ergodicity of $\{Y_n\}$ implies that the time until extintion of $C(\lambda)$ has finite mean.

Observe that $\{Y_n\}$ is irreducible and aperiodic. We use Foster's theorem to show that $\{Y_n\}_{n\geq 0}$ is ergodic for all $\lambda>0$. Consider the function
$f:\mathbb{N}\rightarrow \mathbb{R}^+$ defined by $f(i)=i+1$,  $\epsilon>0$ and the set
$$A:=\left\{i\in \mathbb{N}: \frac{\lambda}{1+\lambda} -\frac{i+1}{2(1+\lambda)} >-\epsilon\right\}.$$ 

For $\lambda>0,$ the set $A$ is finite. Moreover we have that \\

\noindent $\begin{array}{lll}
\bullet \, \mathbb{E}[f(Y_{n+1})-f(Y_{n}) \ | \ Y_n=i]\
&=&[f(i+1)-f(i)]P_{i,i+1}+ \displaystyle\sum_{j=0}^{i-1}[f(j)-f(i)]P_{i,j}  \\ 
&=& \displaystyle\frac{\lambda}{1+\lambda}-\frac{i+1}{2(1+\lambda)}
\leq  -\epsilon \quad \text{for } i\notin A.
\end{array}$\\ 

\noindent$
\bullet \ \mathbb{E}[f(Y_{n+1})|\ Y_n=i] = f(i+1)P_{i,i+1} +\displaystyle\sum_{j=0}^{i-1}f(j)P_{i,j} \le \frac{\lambda(i+2)}{1+\lambda}+\frac{i+1}{2(1+\lambda)} < \infty \text{ for } i\in A. 
$\\

\noindent It follows from Foster's theorem that $\{Y_n\}$ is ergodic and that concludes the proof.
\end{proof}
		
\subsection{Growth models with dispersion and spatial restriction}

\begin{proof}[Proof of Theorem \ref{condicao_extincao_uni_indepen}]
From Machado \textit{et al.}~\cite[Proposition 4.3]{MRV2018}, we have that
\[\psi_d < 1 \quad \text{if and only if} \quad \mathbb{E}\left[ \left(\dfrac{d-1}{d}\right)^N \right] < \dfrac{d-1}{d}.\]
Furthermore, by Proposition~\ref{distribuiçao_individuos}, 
	\begin{equation*}
		\mathbb{E}\left[ \left(\dfrac{d-1}{d}\right)^N \right]=\dfrac{d\ln\left(\dfrac{\lambda+d}{d}\right)}{\lambda}.
	\end{equation*}
	Therefore, $\psi_d<1$ if and only if $\dfrac{d^2}{d-1}\ln\left(\dfrac{\lambda+d}{d}\right)<\lambda$. 
	
\end{proof}

\begin{proof}[Proof of Theorem~\ref{uniforme_independente2}]
	The condition for survival is an application of the Theorem \ref{condicao_extincao_uni_indepen} for $d=2$. In this case, 
	\begin{equation*}
		\psi_2<1 \text{ if and only if }  4\ln\left(1+\frac{\lambda}{2}\right)<\lambda.
	\end{equation*}

To obtain the extinction probability, we consider Proposition 4.4 from Machado~\textit{et al.}~\cite{MRV2018}. Since \(d=2\), \(\psi_2\) is the smallest non-negative solution to the equation
\[
\sum_{y=0}^2 s^y\binom{2}{y}\sum_{n=y}^\infty \dfrac{T(n,y)}{2^n}\mathbb{P}(N=n)=s,
\]
where
	\begin{equation}\label{T(n,y)}
		T(n,y)=\sum_{i=0}^y (-1)^i\binom{y}{i}(y-i)^n.
	\end{equation}
  After substitution and simplification, we obtain
	\begin{equation*}
		\mathbb{P}(N=0)+2s\sum_{n=1}^\infty \dfrac{1}{2^n}\mathbb{P}(N=n)+s^2\sum_{n=2}^\infty\dfrac{2^n-2}{2^n}\mathbb{P}(N=n)=s.
	\end{equation*}
  which can be rewritten as
\begin{equation*}
		(1-s)^2\,\mathbb{P}(N=0)+2s(1-s)\,\mathbb{E}\left[\left(\frac{1}{2}\right)^N\right]=s(1-s).
	\end{equation*}
  
	Consequently, either $s=1$ or
	\begin{equation*}
		s=\dfrac{\mathbb{P}(N=0)}{1+\mathbb{P}(N=0)-2\,\mathbb{E}\left[\left(\frac{1}{2}\right)^N\right]}=\dfrac{\ln(\lambda+1)}{\lambda+\ln(\lambda+1)-4\ln\left(1+\lambda/2\right)}.
	\end{equation*}
	where the final equality is derived from Proposition \ref{distribuiçao_individuos}.

\end{proof}

\begin{proof}[Proof of Theorem~\ref{uniforme_independente3}]
	The condition for survival is an application of the Theorem \ref{condicao_extincao_uni_indepen} for $d=3$. In this case,  
	\begin{equation*}
		\psi_3<1 \text{ if and only if } \dfrac{9}{8} \ln\left(1+\dfrac{\lambda}{3}\right)<\lambda.
	\end{equation*}

 To obtain the extinction probability, we consider Proposition 4.4 from Machado~\textit{et al.}~\cite{MRV2018}. Since \(d=3\), \(\psi_3\) is the smallest non-negative solution to the equation 
	\begin{equation*}
		\sum_{y=0}^3 s^y\binom{3}{y}\sum_{n=y}^\infty \dfrac{T(n,y)}{3^n}\mathbb{P}(N=n)=s,
	\end{equation*}
where $T(n,y)$ is given by equation (\ref{T(n,y)}).
	
After substitution and simplification, we obtain  
\small{
\[
\mathbb{P}(N = 0)
+ 3s \sum_{n = 1}^{\infty} \tfrac{1}{3^n} \mathbb{P}(N = n)
+ 3s^2 \sum_{n = 2}^{\infty} \tfrac{2^n - 2}{3^n} \mathbb{P}(N = n)
+ s^3 \sum_{n = 3}^{\infty} \tfrac{3^n - 3 \cdot 2^n + 3}{3^n} \mathbb{P}(N = n)
= s,
\]
}
which can be rewritten as  
\[
(1-s)^3\mathbb{P}(N=0)
+3s(1-s)^2\mathbb{E}\!\left[\!\left(\tfrac{1}{3}\right)^{\!N}\!\right]
+3s^2(1-s)\mathbb{E}\!\left[\!\left(\tfrac{2}{3}\right)^{\!N}\!\right]
=s(1-s^2).
\]
Using Proposition~\ref{distribuiçao_individuos}, we obtain  
\[
\frac{(1-s)^3}{\lambda}\ln(1+\lambda)
+\frac{9s(1-s)^2}{2\lambda}\ln\!\left(1+\frac{2\lambda}{3}\right)
+\frac{9s^2(1-s)}{\lambda}\ln\!\left(1+\frac{\lambda}{3}\right)
=s(1-s^2),
\]
whose roots are given by \( s = 1 \) and  
\[
s = \frac{1 + 2a - 9b \pm \sqrt{(1 - 9b)^2 + 4a(2 - 9c)}}{2(a - 9b + 9c - 1)},
\]
with  
\[
a = \frac{\ln(1+\lambda)}{\lambda}, \quad
b = \frac{1}{2\lambda}\ln\!\left(1+\frac{2\lambda}{3}\right), \quad
c = \frac{1}{\lambda}\ln\!\left(1+\frac{\lambda}{3}\right).
\]
     
\end{proof}

 \begin{proof}[Proof of Theorem~\ref{esperanca_indepen}] 
First, observe that each colony in $C_d(\lambda)$ survives for an exponentially distributed time with rate $1$ until it is struck by a catastrophe. Each individual that survives the catastrophe then randomly selects a neighboring vertex and attempts to establish a new colony there. When the number of survivors is $n$, the probability that $k \le \min\{d,n\}$ vertices become colonized is
\[
\frac{T(n,k)}{d^n} \binom{d}{k},
\]
where
\[
T(n,k) = \sum_{i=0}^{k} (-1)^i \binom{k}{i}(k-i)^n
\]
is the number of surjective functions $f:A \to B$ with $|A| = n$ and $|B| = k$. 

Let $W_t$ denote the number of colonies at time $t$ in the model $C_d(\lambda)$. Note that $W_t$ is a continuous-time branching process with $W_0 = 1$. Thus, each colony in $W_t$, just before dying, produces $k \le d$ offspring colonies with probability $p_k$ given by
\begin{equation}\label{p_k}
p_k =
\begin{cases}
\mathbb{P}(N=0), & k=0, \\[1em]
\displaystyle \sum_{n=k}^\infty \binom{d}{k} \frac{T(n,k)}{d^n} \mathbb{P}(N=n), & 1 \le k \le d-1, \\[1em]
1 - \displaystyle\sum_{j=0}^{d-1} p_j, & k=d.
\end{cases}
\end{equation}

Moreover, $\tau_d = \inf\{t > 0 : W_t = 0\}$.

\noindent$\bullet$ For $d = 2$, substituting into~(\ref{p_k}) and applying Proposition~\ref{distribuiçao_individuos}, we obtain

\begin{align*}
p_0 &= \frac{\ln(\lambda + 1)}{\lambda}, \\[0.5em]
p_1 &= 2\!\left( \mathbb{E}\!\left[\left(\tfrac{1}{2}\right)^{N}\right] - \mathbb{P}(N = 0) \right)
      = \frac{2}{\lambda}\,\ln\!\left(\frac{(\lambda + 2)^2}{4(\lambda + 1)}\right), \\[0.5em]
p_2 &= 1 - \frac{1}{\lambda}\,\ln\!\left(\frac{(\lambda + 2)^4}{16(\lambda + 1)}\right).
\end{align*}

Furthermore, the condition $4\ln\left(1+\frac{\lambda}{2}\right)>\lambda$ is equivalent to $p_1 + 2p_2 < 1$. Thus, by Junior~\textit{et al.}~\cite[Lemma~4.1($i$)]{JMR2023}, we obtain
\[
\mathbb{E}[\tau_2]
    = \frac{1}{p_2}\,\ln\!\left(\frac{p_0}{p_0 - p_2}\right),
\]
which yields the statement of Theorem~\ref{esperanca_indepen} with $\alpha = p_0$ and $\beta = p_2$.

If $4\ln\left(1+\frac{\lambda}{2}\right)=\lambda$, then $p_1 + 2p_2 = 1$. In this case, by Junior~\textit{et al.}~\cite[Lemma~4.1($i$)]{JMR2023}, we have $\mathbb{E}[\tau_2]=\infty$.

 \noindent$\bullet$ For $d = 3$, substituting into~(\ref{p_k}) and applying Proposition~\ref{distribuiçao_individuos}, we obtain
\begin{align*}
p_0 &= \frac{\ln(\lambda + 1)}{\lambda}, \\[0.5em]
p_1 &= 3\!\left( \mathbb{E}\!\left[\tfrac{1}{3^{N}}\right] - \mathbb{P}(N = 0) \right)
      = \frac{3}{2\lambda}\,\ln\!\left(\frac{(2\lambda + 3)^3}{27(\lambda + 1)^2}\right), \\[0.5em]
p_2 &= 3\!\left( \mathbb{E}\!\left[\tfrac{2^{N} - 2}{3^{N}}\right] + \mathbb{P}(N = 0) \right)
      = \frac{3}{\lambda}\,\ln\!\left(\frac{(\lambda + 1)(\lambda + 3)^3}{(2\lambda + 3)^3}\right), \\[0.5em]
p_3 &= 1 - \frac{1}{2\lambda}\,\ln\!\left(\frac{(\lambda + 3)^{18}(\lambda + 1)^2}{27^3(2\lambda + 3)^9}\right).
\end{align*}

Furthermore, the condition $ \dfrac{9}{2} \ln\left(1+\dfrac{\lambda}{3}\right)>\lambda$ is equivalent to $p_1+2p_2+3p_3<1$. Thus, by Junior~\textit{et al}~\cite[Lemma 4.1($ii$)]{JMR2023}, we obtain
	\begin{align*}
		\mathbb{E}[\tau_3]&=\dfrac{1}{\sqrt{4p_0p_3+(p_2+p_3)^2}}\ln\left[ \dfrac{2p_0-p_2-p_3+\sqrt{4p_0p_3+(p_2+p_3)^2}}{2p_0-p_2-p_3-\sqrt{4p_0p_3+(p_2+p_3)^2}} \right],
	\end{align*}
which yields the statement of Theorem~\ref{esperanca_indepen} with $\alpha = p_0,$ $\theta=p_2+p_3$ and $\gamma = p_3$.

If $ \tfrac{9}{2} \ln\left(1+\tfrac{\lambda}{3}\right)=\lambda$, we have that $p_1+2p_2+3p_3=1$. Thus, by Junior~\textit{et al}~\cite[Lemma 4.1($ii$)]{JMR2023}, $\mathbb{E}[\tau_3]=\infty$.
\end{proof}

\subsection{Growth models with dispersion but no spatial restriction}

\begin{proof}[Proof of Theorem~\ref{th:comdisptime}]
The number of colonies in the process $C_*(\lambda)$ can be modeled as a branching process in which each colony survives for an exponentially distributed time with rate~1 and, just before dying, produces a random number of offspring colonies according to the distribution of the random variable~$N$. The distribution, expectation, and generating function of~$N$ are given in Proposition~\ref{distribuiçao_individuos}. 

Therefore, the process $C_*(\lambda)$ survives if and only if $\lambda > 2$. Moreover, the extinction probability~$\psi_*$ is the smallest nonnegative solution of
\[
\frac{\ln[1 + \lambda(1 - y)]}{\lambda(1 - y)} = y.
\]
Furthermore, according to Narayan~\cite{Narayan1982}, if $\lambda \le 2$, the mean extinction time is given by
\[
\mathbb{E}[\tau_*] = \int_0^1 \frac{1 - y}{f(y) - y}\,dy,
\]
where
\[
f(y) = \frac{\ln[1 + \lambda(1 - y)]}{\lambda(1 - y)}, \qquad 0 < y < 1.
\]
\end{proof}

%%%%%%%%%%%%%%%%%%%%%%%%%%%%%%%%%%%%%%%%%%

%\section{Acknowledgments} The authors are thankful for the two anonymous referees  for a careful reading, many suggestions and corrections that helped to improve the  paper.

%\section{Acknowledgments} 

\end{document}